\documentclass{article}
\usepackage{color}
\usepackage{hyperref}
\usepackage{geometry}
 \usepackage[acronym]{glossaries}
\usepackage{optidef}
\usepackage{amssymb,bm,amsthm}
\usepackage{mathtools}
  \usepackage{textcomp}
  \usepackage{alphalph}
  \usepackage{units}  
\usepackage{graphicx}
\usepackage{wrapfig}
\usepackage{multirow}
\usepackage{tikz}
 \usetikzlibrary{decorations.pathreplacing,calc,positioning,patterns}

\DeclareMathOperator*{\argmin}{argmin}
\DeclareMathOperator*{\argmax}{argmax}
\DeclareMathOperator*{\maximize}{maximise}

\DeclarePairedDelimiter{\group}{(}{)}

\newcommand*{\coacts}{{\mathcal{Y}}}

\newcommand*{\p}{\ensuremath{P}}
\newcommand*{\lp}{\underline{\p}}
\newcommand*{\up}{\overline{\p}}

\newcommand*{\gambles}[1]{\mathcal{G}(#1)}

\renewcommand*{\iff}{\Leftrightarrow}

\usepackage[acronym]{glossaries}
\usepackage[
backend=biber,
style=numeric,
sorting=ynt
]{biblatex}
\usepackage{optidef}
\makeglossaries

\newglossaryentry{entryOne}
{
        name=Glossary Entry,
        description={Glossary entries are used to provide definitions for words in your document}
}    
\providecommand{\keywords}[1]
{
  \small	
  \textbf{\textit{Keywords---}} #1
}

\title{Path Planning Problem under non-probabilistic Uncertainty}
\author{Keivan Shariatmadar$^1$\\
{\small $^1$ Department of Computer Science, M-Group, Campus Bruges, KU Leuven, Belgium }
}
\date{\small\it keivan.shariatmadar@kuleuven.be}
\newtheorem{proposition}{Proposition}

\begin{document}
\maketitle

\begin{abstract}
This paper considers theoretical solutions for path planning problems under non-probabilistic uncertainty used in the travel salesman problems under uncertainty. The uncertainty is on the paths between the cities as nodes in a travelling salesman problem. 
There is at least one path between two nodes/stations where the travelling time between the nodes is not precisely known. This could be due to environmental effects like crowdedness (rush period) in the path, the state of the charge of batteries, weather conditions, or considering the safety of the route while travelling. In this work, we consider two different advanced uncertainty models (i) probabilistic--precise uncertain model: Probability distributions and (ii) non-probabilistic---imprecise uncertain model: Intervals. We investigate what theoretical results can be obtained for two different optimality criteria: maximinity and maximality in the travelling salesman problem. 
\end{abstract} \hspace{10pt}


\footnotetext{\small\textbf{Abbreviations:} AGV, automated guided vehicles; LP, linear programming; LPUU, LP under uncertainty; TSP, travelling salesman problem; CDF, cumulative distribution function; PDF, probability density function}

\section{Introduction}\label{sec-intro}

Automated guided vehicles or AGVs are mobile robots or autonomous vehicles. They perform transportation tasks in many (industrial) applications to transport goods/materials around large (industrial) buildings, such as a factory or warehouses. For instance, they transport materials/goods from warehouses to material handling in assembly lines, to a pharmacy, to a supermarket, or a bookstore see e.g., \cite{Bostelman2016}. Not only one robot but a fleet of mobile robots cooperate to perform an efficient transport of materials, goods, or products. 
One of the main applications in path planning is the travel salesman problem which is widely used in automated guided vehicle problems. In resource management, the current industrial automated guided vehicle systems operate under known paths or routes to control the whole deterministic fleet of automated guided vehicles. In realistic manufacturing systems, there are usually uncertainties regarding the availability of accessible paths or routes to travel between nodes, and even worse, in the presence of uncertainty, the uncertainty is not always unique and varies (advanced uncertainty). In this work, we consider the travelling salesman problem under an advanced (non)probabilistic uncertainty model. The uncertainty is between the nodes e.g., task locations or charging stations in a factory with at least one automated guided vehicle. 
Since 1954, \cite{ContainerTerminal} AGVs are widely used in logistical environments all over the globe and are now on the rise in a lot of other areas and applications. Current AGV systems--mobile robots--incorporate high-technological features to manage the (whole) fleet of robots in more efficient, safer, and \emph{deterministic} ways. Although, there are some features which are unchanged for years see e.g., \cite{Changjoo2015,Tgal1979,Ward1990,LIN2007,DeRyck2020b}.  
An interesting problem in (deterministic) resource management systems in current industrial systems are the decentralised AGV system. The AGVs perform a specified sequence of transportation tasks between different locations in a known deterministic operating area/paths. The current optimised path planning is based on deterministic paths and routes between nodes or charging stations \cite{Matthias2019}. 
In current manufacturing processes, this has an impact if a path or a route is crowded or not available for a while. For example, the path is crowded or affected by weather conditions. Especially in rush hours, some routes are loaded with other agents like humans/operators, vehicles or other AGVs. In other words, the availability of the paths between the nodes/stations is not deterministic and it is uncertain. 
In this paper, we will represent the Travelling Salesman Problem\cite{Zhang2019,Pecin2017a,Pecin2017b} (TSP)--where there is uncertainty at least in one path between two nodes/stations. The travelling salesman approach is widely used in the AGV problem. We assume there is uncertainty in the travelling time (distance) between the cities. The uncertainty is given in two formats: (i) probabilistic, such as distribution functions, and (ii) non-probabilistic, such as interval (advanced case). We will use an uncertain TSP to illustrate our approaches in both cases. We convert the uncertain TSP to a decision problem via imprecise decision theory which is one of our past works\cite{keivan2020Epsi}. In the advanced uncertainty case, we will give the solutions\footnote{In probabilistic case (i), both maximin and maximal solutions coincide.} under two decision optimality criteria---maximinity and maximality.
 
Generally, the potential applications of this paper can be extended to larger scopes, such as delivery cars, trucks, or autonomous/semi-autonomous (electrical) vehicles where, for instance, those trucks or cars are tasked to travel/move in a path of pickup and drop off points where the path (availability) is uncertain. Our approach can be implemented in these situations in a way that drivers/fleet managers can be advised to choose a better path on their uncertain trajectory to get time-energy optimal paths. In other words, the approach presented here can be implemented in every resource-based transportation system because of its simplicity and generality.
 
The paper is organised as follows. In Section \ref{revise}, a short recap and definitions of \emph{coherent lower and upper previsions}, as well as the \emph{imprecise decision theory} and the current literature on resource management in AGV systems, are given, for more info about imprecise probability theory, see Section. \ref{Annex}. In Section $\ref{opt:utsp}$, the theoretical results are discussed, together with solutions of Uncertain TSP\footnote{The difficulty of the theoretical solutions is reduced to well-posed classical problems by illustrating imprecise decision theory in this paper.}. Section $\ref{epsinum}$ discusses comparisons with reviews of the existing literature to compare our method and results. Section $\ref{sec:concl}$ discusses the conclusions and future work.

\section{Overview and definitions of the theoretical framework}
\label{revise}
This section helps the reader to grasp a quick view of the most important concepts of imprecise probability\footnote{For more details, we propose consulting the reference book \cite{Walley-1991} as well as Section \ref{Annex}.} and uncertain linear programming (LP) problems. To make the concept of coherent lower previsions in imprecise probabilities theory more understandable, a short description will be given in Section. $\ref{cohsec}$. In Section. $\ref{imdecsec}$ the concept of imprecise decision theory is explained. In the next Section \ref{sec:lp} we will talk about LP problems and generic LP under uncertainty (LPUU) problems. 

\subsection{Linear programming problems}\label{sec:lp}
In previous works \cite{keivan2020Epsi,Keivan-1225787,Erik-2090091,Erik-2030199} analytical methodologies and some numerical solutions for the LP under advanced $\epsilon-$contamination, interval, fuzzy sets, and probability-box uncertainty models have been discussed. There are many applications for LP under uncertainty (LPUU) problems, some are discussed by \emph{Dantzig} in Example 2. \cite{Dantzig-1955}, which is usually about finding the \emph{minimum expected cost } for example, the lowest cost of diet in a \emph{Nutrition problem}. Here are some interesting broader applications of optimisation under uncertainty: Optimisation under uncertainty in Artificial Intelligence, Operation of reservoirs, Generation of electrical power, Portfolio selection and optimisation, Inventory management, Facility planning, Stabilisation of mechanisms, Pollution control and Analysis of biological systems \cite{Rockafellar-2001}. In the next two subsections \ref{sec:lpmath} and \ref{sec:lpuumath} we will discuss a mathematical overview of the LP and LPUU problems, respectively. 

\subsubsection{Mathematical overview of LP problems}\label{sec:lpmath}
A standard or canonical LP problem is expressed as follows \cite{Nemhauser-G-L-1989}: 
\begin{align}\label{lpp}
\maximize &\quad c^Tx \notag\\
\text{such that} &\quad Ax\le b,~ x\ge 0
\end{align}
where \(x\in\mathbb{R}^n\) is an optimisation variable, $A\in\mathbb{R}^{m\times n}$ is the coefficient matrix, $c\in\mathbb{R}^n$ is the objective function coefficient vector and $b\in\mathbb{R}^m$ is the constraint vector. 

\subsubsection{Mathematical overview of LPUU problems}\label{sec:lpuumath}
An LPUU problem is a generalisation of the LP problem where at least one of the parameters (coefficients) of the LP problem is uncertain. There is at least one element of the coefficients\rq{} matrices $(c^T, A, \text{~or~} b)$ of the LP problem \eqref{lpp}, which is uncertain or generally \emph{unknown} i.e., they are not deterministic, we do not know the exact values or the values are not known, precisely. For instance, the only information about the coefficient is just boundaries i.e., lower and upper values in an interval, or some probabilistic information.
The challenge lies in optimising an objective linear function in an unknown domain (set). Usually, we do not know exactly whether the problem is feasible or not because the coefficients of the constraints or the goal function are uncertain. In other words, the problem---maximising a linear function over the unknown set (unknown \emph{feasibility space})---is not well-posed. The generic (standard) uncertain linear programming problem is defined in the following form.
\begin{align}\label{spp}
\maximize &\quad U^Tx \notag\\
\text{such that} &\quad Yx\le Z,~ x\ge 0
\end{align}
where \(x\in\mathbb{R}^n\) is a vector of optimisation variables $x_j$, $U$ is a random vector taking variables $u\in\mathbb{R}^n$, the matrices $Y$ and $Z$ are random variables taking values $y\in\mathbb{R}^{m\times n}$ and $z\in\mathbb{R}^m$, respectively\footnote{\label{uncertaintymodel} We assumed that $y_{ij}$, $z_i$ and $u_j$ the elements of $Y$, $Z$ and $U$ are independent. In this paper, we work with the maximisation operator. Since $\min~U^Tx = -\max~-U^Tx$ therefore all results and proves can be applied and held for the minimisation operator as well.}.

\subsection{Imprecise Uncertainty modelling---Coherent lower and upper previsions}
\label{cohsec}
Suppose\footnote{For the general uncertainty models and a tool---expectation operators or probability measures for $(Y,Z)$ in problem \eqref{spp}---we direct the reader to \cite{Walley-1991}. The uncertainty models---probability measures, intervals, and so on---can all be special cases of uncertainty modelling framework: the theory of coherent lower previsions (and imprecise probability theory in general) see, e.g. \cite{Walley-1991, Miranda-2008}.} an unknown variable $(Y,Z)=:V,$ that takes values $(y,z)=:v$ in a set $\mathbb{R}^{m\times n}\times\mathbb{R}^m=:\mathcal{Y}$ and a decision maker (agent) who wants to make decisions about a problem that is a function of $V$. The uncertainty of the agent working with this $V$ is given using an uncertainty model which allows him to do reasoning about $V$ or a function of $V$ and make the decisions in the problem involving $V$. Classical uncertainty models are unique probability distributions. In this case, we call the uncertain variable $V$ a \emph{random variable}. It has been shown \cite{Whittle-1992} that working with expectation tools is the same as working with probability measures or probability distributions. In this paper, the same terminology of Walley \cite{Walley-1991} is used where the expectation operator is called linear prevision/expectation. A linear prevision is a functional $P$ such that \(P:\mathcal{G(Y)}\longrightarrow\mathbb{R}\)
where $\mathcal{G(Y)}$ is a linear space on $\mathcal{Y}$. The linear prevision $P$ satisfies the following three coherence conditions:

\(\text{(Positivity)}\label{pos-P}~P(g)\ge\inf g;~\text{(Homogeneity)}\label{hom-P}~P(\lambda g) =\lambda P(g);~\text{(Additivity)}\label{add-P}~P(g+h)=P(g)+P(h)\) for all bounded real-valued functions $g,h$ in $\mathcal{G(Y)}$ and all $\lambda\in\mathbb{R}$.
The functional $g=g(V)$ is interpreted as a gamble about $V$ and its linear prevision ($P_V$) as a fair price to exchange this gamble, see e.g., \cite{DeFinetti-1975}. An agent (decision maker) is willing to sell the gamble $g(V)$ for any price higher than $P_V(g)$ and buy it for any lower price. For an event, $A\in\mathcal{Y}$ the linear prevision is also denoted by $P$. In this case, $P$ is called probability measure on the set of all events $2^{\mathcal Y}$, i.e., $P:2^{\mathcal{Y}}\longrightarrow[0,1]$. The relationship between them is given by $P(A):=P(I_A)$ where $I_A\in 2^{\mathcal{Y}}$ is the indicator function of $A$---it takes the value $1$ on $A$ and is $0$ otherwise. The three coherence conditions for a probability measure $P$ are:
\(\text{(Positivity)}\label{pos-PM}~P(A)\ge0;~\text{(Unit Norm)}\label{hom-PM}~P(\mathcal Y)=1;~\text{(Additivity)}\label{add-PM}~P(A\cup B)=P(A)+P(B)\) where $A,B\subseteq\mathcal Y,~A\cap B=\emptyset $. For more information on the linear prevision, we refer the reader to \ref{LP-prop}.

\subsection{Imprecise decision making}
\label{imdecsec}
Consider a case that the agent may choose a decision $x$ between several choices, acts, or decisions in a set $\mathcal X:=\mathbb R^n$, the outcome of each decision is uncertain and is a function of the random variable $V$ taking values $v$. For each possible decision $x$ there is a gain (loss) function $G_x$ on $\mathcal Y$, that is, if a decision is made $x$, then the result of this decision has utility $G_x(v)$ where $v$ is the outcome of the random variable $V$.\\
In this paper, we assume that for each decision x there is a corresponding bounded function (gamble) $G_x$ where \(G_x:\mathcal Y\longrightarrow\mathbb R\). 
If the uncertainty about $V$ is described by a coherent lower prevision $\underline P$, then a binary relation in the set $\mathcal X$ of all decisions can be defined as follows: decision $x_1$ is better than decision $x_2$ and we write $x_1\succ x_2$ if and only if the agent is willing to pay some strictly positive prices to exchange $G_{x_1}$ for $G_{x_2}$, that is,
\begin{equation}
x_1\succ x_2\iff\underline P(G_{x_1}-G_{x_2})>0,
\end{equation}
according to the definition, the relation $\succ$ gives us a very useful interpretation, a strict partial order on $\mathcal{X}$. For instance, when the uncertainty about $V$ is represented by a linear prevision $P$ then from the \ref{add-P} property and equation \eqref{l-env} we have
\begin{equation}
x_1\succ x_2\iff\forall P\in\mathcal{M}(\underline P),~ P(G_{x_1})>P(G_{x_2}),
\end{equation}
which means, the action $x_1$ is better than $x_2$ if and only if $x_1$ has a strictly higher expected utility than $x_2$ for all $P$ that dominate the lower prevision $\underline P$ (which is a similar robustness property). In the next section, two decision criteria are described---maximinity and maximality.

\subsubsection{Maximality}
Consider a case that a decision maker seeks decisions $x$---so-called maximal decisions/solutions---that are undominated in pairwise comparison with all other decisions (partial order), i.e., no decision $z$
is considered better than $x$:
\begin{align}\label{maximality}
x~\text{is maximal}&\iff\nexists z\in\mathcal X,~z\succ x\equiv\forall z\in\mathcal{X},~z\not\succ x\notag\\
 &\iff\forall z\in\mathcal{X},~\underline P(G_z-G_x)\le 0\notag\\
&\iff\forall z\in\mathcal{X},~\overline P(G_x-G_z)\ge 0\equiv\inf_{z\in\mathcal X}\overline P(G_x-G_z)\ge 0.
\end{align}

\subsubsection{$\Gamma$-maximin (Maximinity)} 
Maximin solutions derive from worst-case reasoning, i.e., they are the decisions that have the highest lower expected utility (worst-case scenario),
\begin{equation}\label{gama-maxi}
x~\text{is maximin or gamma maximin}\iff x\in\argmax_{z\in\mathcal{X}}\underline P(G_z)
\end{equation}
Similarly, maximaxity solutions---best-case reasoning/scenario---can be found by simply replacing the lower prevision $\underline P$ with the upper prevision $\overline P$ in \eqref{gama-maxi}.

\begin{proposition}\label{prop.1.}
Any maximin solutions are also in the maximal solutions set.
\end{proposition}
\begin{proof}
 Due to the properties of the upper prevision in Subsection \ref{msadd-lp}, we have:
\[\overline P(G_x-G_z)\ge\underline P(G_x)+\overline P(-G_z)=\underline P(G_x)-\underline P(G_z)\]
\end{proof}
In both the maximinity and maximality criteria, $\argmax_{z\in\mathcal{X}}\underline P(G_z)$ and \(\underset{z\in\mathcal X}{\inf}\overline P(G_x-G_z)\) are functions of $x$ in $\mathcal{X}$, therefore, we need to calculate and find that (i) in maximinity: for which $z\in\mathcal X$ the function---$\underline P(G_z)$---has the highest value, and (ii) in maximality: the function---\(\underset{z\in\mathcal X}{\inf}\overline P(G_x-G_z)\)---is positive or zero\footnote{For further information and details in decision making with imprecise probabilities, we refer to \cite{Troffaes-2007-decision}.}. 

\subsection{Reformulation of LPUU problem as an imprecise decision problem}\label{reformidp}
In Section $\ref{sec-intro}$ we talked about the LP problem \eqref{lpp} which is about maximising a linear function over a (convex) set. By adding uncertainty to the constraints (or the goal function), we have a slightly more difficult problem \eqref{spp} that is not a well-posed problem: maximising a linear objective function over an uncertain set. To obtain a well-posed problem, we reformulate the problem \eqref{spp} to a decision problem under uncertainty, as follows: \\
\underline{First}, we define a utility (gain/loss) function (gamble) $G_x$ for each decision $x\ge 0$, as such: 
\begin{equation}\label{GAIN}
G_x:=(c^Tx-L)I_{Yx\le Z}+L
\end{equation}
$I_{Yx\le Z}(v)$ is an indicator function that is equal to one if $x$ is in the feasibility space and is zero when $x$ is infeasible for any realisations $v=(y,z)$ that the random variable $V=(Y,Z)$ assumes. It is obvious that for each decision $x\ge 0$ and any outcome or realisation $(y,z)$ when $x$ is feasible (or equivalently $I_{Yx\le Z}(v)=1$) then we have a reward equal to $c^Tx$ otherwise we have to be punished with real value $L$ (or equivalently $I_{Yx\le Z}(v)=0$). In other words,
\begin{align}
G_x(y,z)=&
\begin{cases}
c^Tx & yx\le z,\notag\\
L & yx\not\le z.
\end{cases}
\end{align}
$L\in\mathbb{R}$ is small enough and is interpreted as a penalty/punishment value for violating the constraints. A real number, strictly smaller than $\inf c^Tx$ (in maximisation) should be chosen to ensure that breaking the constraint is penalised. If there are several penalty values $L_i\in\mathbb{R}$ for each $i$-th constraint, then we can define $L:=\max_{1\le i\le m}~|L_i|$. With this reformulation, it is clear that both the maximums for the objective function $c^Tx$ and the gain function $G_x$ are equivalent.\\
\underline{Second}, in order to consider the uncertainty about $V$ and quantify it, we use an optimality criterion. In this paper, we consider two optimality criteria: maximinity/worst-case reasoning and maximality/partial ordering which we discussed in Section. $\ref{imdecsec}$.

\subsubsection{Set of maximin and maximal solutions for the LPUU problem}
By applying the maximinity criterion in \eqref{gama-maxi} to the gain/loss function defined in \eqref{GAIN} and two Constant additivity and Positive Homogeneity properties of the lower prevision, we find the set of maximin solutions which is given by, 
\begin{align}\label{MAXIMIN}
\argmax_{x\ge 0}\underline P( G_x)=&\argmax_{x\ge 0}\underline P\bigl((c^Tx-L)I_{Yx\le Z}+L\bigr) \notag\\
=&\argmax_{x\ge 0}(c^Tx-L)\underline P(Yx\le Z).
\end{align}
Applying the maximality criterion in \eqref{maximality} to the gain function defined in \eqref{GAIN}, we have that $x\ge 0$ is maximal if and only if
\begin{equation}\label{MAXIMAL}
\inf_{w\ge 0}\overline P\bigl((c^Tx-L)I_{Yx\le Z}-(c^Tw-L)I_{Yw\le Z}\bigr)\ge 0.
\end{equation}
In the next section, we apply these theoretical solutions to a generic TSP and present two sets of optimal solutions. 

\section{Optimal solutions for uncertain TSP}
\label{opt:utsp}
Generally, a TSP is a specific form of the generic LP problem. To model our approach, we will build further on the well-known TSP model given by \cite{TSP} representation. The TSP is a mathematical graph theory problem used in computer science to model real-world specific optimisation problems such as path planning problems. 
\subsection{Mathematical model of uncertain TSP}
\label{sec:TSPModel}
Assume that $n$ selected cities on a salesman's\footnote{The salesman can be interpreted as an AGV or an agent in different domains, as well.} tour are the vertices set (similar to the set $V$ in a graph). The graph's set of edges $E$ corresponds to the different connections/paths between each city. Since the salesman can travel from any city to another, the graph is complete. In other words, there is an edge between every pair of nodes. For each edge $(i,j)$ on the graph, we associate a binary variable $I_{ij}$,
\[I_{ij}=\begin{cases}
1\quad (i,j)\in E\\
0\quad \text{else}.
\end{cases}\]
Since in the TSP, the edges are undirected then $x_{ij}=x_{ji}$ and it suffices to only include edges with $i<j$ in the problem (mathematical model), as follows.
A salesman has the task of visiting $n$ cities $x_{1},\dots,x_{n}$ starting from an origin city $x_{1}$. The salesman does not know (is uncertain about) all the distances between each city. He has some probabilistic or interval uncertainty models about the distances. He has to visit every city once and has an uncertain distance matrix $U_D=\big[u^d_{ij}\big]$. $U_D$ is random matrix taken values $u^d_{ij}$--the uncertain distance from city $i$ to city $j$. He has to visit all cities once in such a way that the total travelled time is minimised. After he visited all the cities, he has to return to the starting city $x_{1}$. The cost function $g_u(x)$ and the constraint in the TSP can be formulated like an LPUU problem as follows:
\vspace{-1mm}
\begin{align}\label{utsp}
\min_x g_u(x) :=  \min_x\quad &\sum_{i=1}^{n-1} (u^t_{x_{i},x_{i+1}}.I_{ij}) + u^t_{x_{n},x_{1}}\notag\\
\text{such that}\,&\sum_{j\in V}I_{ij}=2,\notag\\ 
&I_{ij}\in\{0,1\},
\end{align}
where, ${x} = (x_{1},\dots,x_{n})$ represents a whole tour which contains every city $x_{i}$ once, and
$u^t_{x_{i},x_{i+1}}$ is the uncertain travelled time\footnote{$u^t_{x_{n},x_{1}}$ is the uncertain travelled time between the last city $x_{n}$ and the starting city $x_{1}$.} between city $x_{i}$ and city $x_{i+1}$. For simplicity, a constant speed $v_{c}$ is assumed, and consequently, the uncertain travel time $u^t_{x_{i},x_{i+1}}$ is defined as follows:  
\begin{equation}
u^t_{x_{i},x_{i+1}} := \frac{u^d_{x_{i},x_{i+1}}}{v_{c}}
\end{equation}
Based on the problem definition in \eqref{utsp}, $g_u$ is the uncertain total time travelled to visit every city once when following the tour ${x}$. The goal of the problem is to find a tour ${x}$ that minimises this objective function $g_u$. This optimised tour ${x}$ contains cities and stations that are optimally chosen for the salesman/AGV to visit. The AGV starts with an already optimised tour\footnote{optimal tour by optimal inserting extra cities (optimal insertion of charging stations and the time it will charge at these stations.) into optimal tour if the robot is not able to finish the tour without charge.} which is received from a central or decentralised task allocation algorithm. This task allocation algorithm outputs a TSP-optimised tour containing all the locations where AGV has to visit (to pick up or drop off a load). Our aim here is to give a classical TSP where the uncertainty of the uncertain TSP \eqref{utsp}, is eliminated. From Section  \ref{reformidp}, we reformulate the problem \ref{utsp} to a decision problem as follows,

\begin{equation}
G_x(u):=\begin{cases}
g_u(x)\quad \sum_{j\in V}I_{ij}=2\\
L\qquad \text{otherwise}
\end{cases}\quad\equiv\qquad\;\big(g_u(x)-L\big)I_{\big\{\sum_{j\in V}I_{ij}=2\big\}}+L
\end{equation}
where $L<\min_{(x,u)} g_u(x)$ is chosen as the real number (punishment) and $I_{\{\cdot\}}$ is an indicator function on the set $\{\cdot\}$. In the next Sections \ref{prob-tsp} and \ref{int-tsp}, we give the maximin and maximal solutions for the uncertain TSP in two uncertainty models---Probability distribution and interval cases. 

\subsection{Probabilistic case---Probability distributions}\label{prob-tsp}
Assume that the uncertainty about the random matrix $U_D=\big[u^d_{ij}\big]$---the durations between two nodes/cities $x_i$ and $x_j$---is given by a probability measure $P$. As explained in Section \ref{opt:lin} and Equation \eqref{LINPRE}, the optimal solutions for the TSP \eqref{utsp} are
\begin{align}\label{maximin-tsp}
\argmin_{x}P(G_x(u))&=\argmin_{x}P\big((g_u(x)-L)I_{\big\{\sum_{j\in V}I_{ij}=2\big\}}\big)\notag\\
&=\argmin_{x}P\big(\big[\sum_{i=1}^{n-1} (u^t_{x_{i},x_{i+1}}.I_{ij}) + u^t_{x_{n},x_{1}}-L\big]I_{\big\{\sum_{j\in V}I_{ij}=2\big\}}\big)\notag\\
&=\argmin_{x}\big[\sum_{i=1}^{n-1} \big(P(u^t_{x_{i},x_{i+1}}).I_{ij}\big) + P(u^t_{x_{n},x_{1}})-L\big]P\big(I_{\big\{\sum_{j\in V}I_{ij}=2\big\}}\big)
\end{align}
where the solution \eqref{maximin-tsp} is a classical (crisp) TSP and it depends on the punishment value $L$.

\subsection{Non-probabilistic case---Intervals}\label{int-tsp}
In this case, for simplicity, we will work with the hypograph model of the linear optimisation problem \eqref{utsp} since the uncertainty is in the goal function. The hypograph model is defined as follows.

\begin{align}\label{hypoutsp}
\min &\quad \sum_{i=1}^{n-1} (u^t_{x_{i},x_{i+1}}.I_{ij}) + u^t_{x_{n},x_{1}}\qquad\equiv\qquad \min\qquad s \notag\\
\text{such that} &\quad \sum_{j\in V}I_{ij}=2,\quad I_{ij}\in\{0,1\}\qquad \text{such that} \quad\sum_{j\in V}I_{ij}=2,\quad I_{ij}\in\{0,1\}\notag\\
&\qquad\qquad\qquad\qquad\qquad\qquad\qquad\qquad\qquad\;\sum_{i=1}^{n-1} (u^t_{x_{i},x_{i+1}}.I_{ij}) + u^t_{x_{n},x_{1}}\le s,\;s\ge 0.
\end{align} 
As it is explained in Section \ref{vacsec}, the uncertainty about $U_D=\big[u^d_{ij}\big]$ is modelled by coherent lower prevision $\underline P_U$ on $A\subseteq\mathbb{R}^n$ for a given $h=h(U)~\text{on}~A$ as, 
\begin{equation}
\underline P_U(h):=\min_{u\in{A}}~h(u)~~\text{equivalently,}~~\overline P_U(h):=\max_{u\in{A}}~h(u)
\end{equation} 
where \(A:=\bigtimes_{k=1}^{n} A_{k};\;A_{k}:=[\underline y_{k},\overline y_{k}].\)

\subsubsection{Maximin Solutions}
 From equation \eqref{maximinsol}, the optimal solutions for the hypograph model \eqref{hypoutsp} are
\begin{align}
\argmax_{x}\underline P(G_x(u))=\argmax_{x\in\underline A}\underline P\big(s\big)
\end{align}
where $\underline A$  is inner feasibility space,
\[\underline A:=\Big\{\sum_{i=1}^{n-1} (\underline u^t_{x_{i},x_{i+1}}.I_{ij}) + \underline u^t_{x_{n},x_{1}}\le s,\;s\ge 0~\wedge~ \sum_{j\in V}I_{ij}=2,\quad I_{ij}\in\{0,1\} \Big\}.\]
where the maximin solution is a classical (crisp TSP) LP problem and it is not dependent on the punishment value $L$.

\subsubsection{Maximal Solutions}
 From equation \eqref{maximalsol}, the optimal solutions for the hypograph model \eqref{hypoutsp} are all $x$'s, in the outer feasibility space $\overline A$,
\[\overline A:=\Big\{\sum_{i=1}^{n-1} (\overline u^t_{x_{i},x_{i+1}}.I_{ij}) + \overline u^t_{x_{n},x_{1}}\le s~\wedge~\sum_{j\in V}I_{ij}=2,\quad I_{ij}\in\{0,1\}~\wedge~s\ge\max(x_m, 0) \Big\},\]
where $x_m$ is the maximin solution. Again here the maximal solutions are not dependent on $L$ and the set $\overline A$ is convex (feasibility space).

\section{Comparison with related literature---alternative existing solutions and examples}
\label{epsinum}
Let\rq{}s now contrast the presented methods with some closely related approaches from the literature on two types of uncertainty models in the TSP and LP problem.
\begin{description}
\item[Changjoo and Dylan]\cite{Changjoo2015} have studied a sensitivity analysis in a multi-robot task allocation problem where the coefficient of the cost function varies between bounded intervals (non-probabilistic case). They gave two algorithms which are exhaustive searches to analyse the sensitivity and not quantify the uncertainty--e.g., in an indeterministic problem, how to make the decision when either availability or traffic on a path is uncertain? 
There are some other works where only the sensitivity of the problem---and not uncertainty quantification in indeterministic problems---is considered \cite{Tgal1979}.
\item[Ward and Richard]\cite{Ward1990} gave approaches to sensitivity analysis in linear programming. Here our method considers reasoning about uncertainty in the coefficient matrix in constraints or the goal function under two different uncertainty models.
\item[Lin and Wen]\cite{LIN2007} wrote about the sensitivity analysis of objective function coefficients of an assignment problem. The focus is only on the special case of an LP problem and only on the sensitivity of the goal function.
\item[Soyster]\cite{Soyster-1973} investigates so-called \emph{inexact linear programming problems:}
\begin{align}
\max &\quad c^Tx \notag\\
\text{such that} &\quad Yx\le b,~x\ge 0
\end{align}
where $c,x\in\mathbb R^n$, $b\in\mathbb R^m$, and the uncertain matrix~$Y$ is interval relative to a convex subset $K:=\bigtimes_{j=1}^n K_j$ of matrices, where each $K_j$ is a convex subset of $\mathbb R^m$. This paper uses the maximinity decision criterion, so this is a special case of Section.~$\ref{maximin-vacuous}$ where, in Equation~\eqref{maxisol-vac} we have now concluded that $\underline{K}=\{x\ge0: \overline yx\le b\}$, where $\overline y$ is the matrix with components $\overline y_{ij}:=\max_{y\in K}y_{ij}$. In Section. $\ref{sec:lpuumath}$ we gave a very general case that the uncertainty can be represented in all coefficients. However, in many optimisation problems involving interval-based uncertainty---formalised by vacuous previsions in this paper---the work above fits in the large body of literature.  Most of the work uses maximinity as the optimality criterion; The results and methods we derive in this paper for the maximality case in Section.~$\ref{maximal-vacuous}$ are novel and make the state of the art richer.
\item[Jamison and Lodwick]\cite{Jamison-Lodwick-2001} worked on \emph{Fuzzy linear programming problems}\footnote{Working with TSP under Fuzzy sets is our current research, however here we emphasise the selection of penalty value $L$ compare to our simple approach.}. Their approach is based on the approach by Bellman and Zadeh \cite{Bellman-Zadeh-1970}. We translate and explain the approach for this discussion relevant to our context. As we present in Section. $\ref{reformidp}$, their idea is to first move the original problem into an unconstrained optimisation problem: \(\maximize~ c^Tx-L_{yx\not\le z}I_{yx\not\le z}\), using, however, a variable penalty factor $L:=L_{yx\not\le z}$, where $L$ is a function of $x$ on $\mathbb R^{n}$ that gives the larger a penalty the more severely the constraint is broken. Then a fuzzy number---a possibility distribution on~$\mathbb R$---$\tilde{G}_{x} := c^Tx-L_{\tilde yx\not\le \tilde z}I_{\tilde yx\not\le\tilde z}(x)$ is associated with every $x$ in $\mathbb R^{n}$ using the extension principle; these can be seen as fuzzy gains. They define the solutions of the problem to be those $x$ with a maximal midpoint average: $\frac{1}{2}\int_0^1\group[\big]{\max G_{x}(t)-\min G_{x}(t)}\text{d}t,$ where $G_{x}(t):=\{\alpha\in\mathbb R:\tilde{G}_{x}(\alpha)\ge t\}$ is the level set at $t$. This optimality criterion lies qua idea in between maximinity and maximaxity, but not qua execution, as how we use possibility distributions to express uncertainty differs markedly.
\end{description}

\section{Conclusions}\label{sec:concl}
In this paper, we looked at a general uncertain TSP with uncertain available paths/routes between cities/stations and investigated how uncertainty can be dealt with parameters in the constraints as well as in the goal function. We have presented theoretical solutions for the generic uncertain TSP with probability distributions and interval uncertainty models. We have discussed the way to reformulate the generic uncertain TSP into a well-posed decision problem according to the probabilistic and interval cases. We have then modelled the uncertainty by using coherent lower (and upper) previsions and proposed theoretical solutions for the decision problem in two typical optimality criteria---maximinity and maximality.\\
By solving the decision problems we have obtained $(i)$ in maximinity: classical linear optimisation problems (without uncertainty), and $(ii)$ in maximality: classical (convex) feasibility problems. Using maximinity results is a less complicated mathematical problem concerning the maximality criterion. In summary,
\vspace{1mm}
\begin{center}
{
	 \begin{tabular}{c|c|c}
	  \textbf Model &  \textit{Maximin theoretical solutions} & \textit{Maximal theoretical solutions} \\
	\hline 
	\multirow{1}{*}{\textit{Probability distributions}}  & classical LP problem & classical LP problem \\	
	\hline
	\multirow{1}{*}{\textit{Interval}}  & classical LP problem & convex feasibility problem  	
	\end{tabular}
	}
	\end{center}\vspace{1mm}
To calculate or implement these theoretical solutions under any numerical applications, any classical linear optimisation problem solver or feasibility problem solver could be used as follows
\vspace{1mm}
\begin{center}
{
	 \begin{tabular}{c|c|c}
	  \textbf Model &  \textit{Maximin calculation} & \textit{Maximal calculation} \\
	\hline 
	\multirow{1}{*}{\textit{Probability distributions}}  & any LP solver & any LP solver \\
	\hline
	\multirow{1}{*}{\textit{Interval}}  & any LP solver & any convex feasibility space solver  
	\end{tabular}
	}
	\end{center}\vspace{1mm}
The paper proposes a way to solve uncertain AGV systems in maximality and maximinity criteria for both uncertainty models. Our next intent is to work towards efficient algorithms (computer codes) to implement numerical solutions to our theoretical results and compare the complexity with the existing algorithms as well as to extend our approaches to the multi-AGV systems and more advanced uncertainty models such as Fuzzy sets, or Probability Box.

\section*{Acknowledgments}
This work is supported by the M-Group, at campus Bruges--KU Leuven. The author thanks all support from Matthias De Ryck. 

\subsection*{Financial disclosure}
MATLAB Licence (Number: 919019 for Version: 9.5.0.1067069 (R2018b) Update 4)

\subsection*{Conflict of interest}
The authors declare that they have no potential conflict of interest.

\section*{Supporting information}
None reported.

\medskip

\printglossary
\printbibliography 

@ARTICLE{Troffaes-2007-decision,
  author = {Troffaes, M. C. M.},
  title = {Decision making under uncertainty using imprecise probabilities},
  journal = {Int. J. Approx. Reason.},
  year = {2007},
  volume = {45},
  pages = {17--29},
  doi = {10.1016/j.ijar.2006.06.001}
}

@article{keivan2020Epsi,
author = {Shariatmadar, Keivan and De Ryck, Matthias and Driesen, Kristof and Debrouwere, Frederik and Versteyhe, Mark},
title = {CMMSE: Linear programming under $\epsilon$-contamination uncertainty},
journal = {Computational and Mathematical Methods},
volume = {2},
number = {2},
pages = {e1077},
doi = {https://doi.org/10.1002/cmm4.1077},
url = {https://onlinelibrary.wiley.com/doi/abs/10.1002/cmm4.1077},
eprint = {https://onlinelibrary.wiley.com/doi/pdf/10.1002/cmm4.1077},
note = {e1077 cmm4.1077},
year = {2020}
}

@book{Walley-1991,
  title={Statistical Reasoning with Imprecise Probabilities},
  author={Walley, P.},
  isbn={9780412286605},
  lccn={lc90042753},
  series={Chapman \& Hall/CRC Monographs on Statistics \& Applied Probability},
  url={https://books.google.be/books?id=Nk9Qons1kHsC},
  year={1991},
  publisher={Taylor \& Francis}
}

@ARTICLE{Miranda-2008,
  author = {Miranda, Enrique},
  title = {A survey of the theory of coherent lower previsions},
  journal = {International Journal of Approximate Reasoning},
  year = {2008},
  pages = {628--658},
  volume = {48}
}

@ARTICLE{Walley-1999,
  author = {Walley, Peter and {d}e Cooman, Gert},
  title = {Coherence of rules for defining conditional possibility},
  journal = {International Journal of Approximate Reasoning},
  year = {1999},
  volume = {21},
  pages = {63--107},
  doi = {10.1016/S0888-613X(99)00007-9}
}

@article{Dantzig-1955,
      title = {Linear Programming under Uncertainty},
     author = {Dantzig, George B.},
     journal = {Management Science},
     volume = {1},
     number = {3/4},
     pages = {197-206},
     url = {http://www.jstor.org/stable/2627159},
     ISSN = {00251909},
     language = {English},
     year = {1955},
     publisher = {INFORMS}
    }

@Article{Soyster-1973,
  title = {Convex Programming with Set-Inclusive Constraints and Applications to Inexact Linear Programming},
  author = {A. L. Soyster},
  journal = {Operations Research},
  pages = {1154--1157},
  volume = {21},
  number = {5},
  month = {9--10},
  year = {1973},
  url = {http://www.jstor.org/stable/168933}
}

@Article{ Bellman-Zadeh-1970,
  title = {Decision-making in a fuzzy environment},
  author = {R. E. Bellman and L. A. Zadeh},
  journal = {Management Science},
  pages = {B-141--B-164},
  volume = {17},
  number = {4},
  year = {1970},
  doi = {10.1287/mnsc.17.4.B141}
}

@Article{ Jamison-Lodwick-2001,
  title = {Fuzzy linear programming using a penalty method},
  author = {K. David Jamison and Weldon A. Lodwick},
  journal = {Fuzzy Sets and Systems},
  pages = {97--110},
  volume = {119},
  number = {1},
  year = {2001},
  doi = {10.1016/S0165-0114(99)00082-2}
}

@article{deCooman20111911,
title = "Independent natural extension",
journal = "Artificial Intelligence",
volume = "175",
number = "12-13",
pages = "1911 - 1950",
year = "2011",
issn = "0004-3702",
doi = "10.1016/j.artint.2011.06.001",
url = "http://www.sciencedirect.com/science/article/pii/S0004370211000737",
author = "Gert de Cooman and Enrique Miranda and Marco Zaffalon"
}

@BOOK{Nemhauser-G-L-1989,
  title = {Handbooks in Operations Research and Management Science, 1: Optimization},
  publisher = {North-Holland/Elsevier},
  year = {1989},
  author = {Nemhauser, G.L. and Rinnooy Kan, A.H.G. and Todd, M.J.},
  location = {Amsterdam},
  edition = {2nd},
     isbn = {978-0-444-87284-5}
}

@BOOK{DeFinetti-1975,
  title = {Theory of Probability},
  publisher = {Wiley},
  year = {1974--1975},
  author = {de Finetti, Bruno},
  note = {2 Volumes}
}

@Book{Whittle-1992,
  title = {Probability via Expectation},
  author = {Peter Whittle},
  edition = {3rd},
  year = {1992},
  publisher={Springer}
}

@inproceedings{Keivan-1225787,
  author       = {Shariatmadar, Keivan and Quaeghebeur, Erik and De Cooman, Gert},
  booktitle    = {ISIPTA '11: Program \& Abstracts},
  language     = {eng},
  organization = {},
  location     = {Innsbruck, Austria},
  pages        = {31},
  title        = {Linear programming under vacuous and possibilistic uncertainty},
  year         = {2011}
}

@article{Erik-2030199,
  author       = {Quaeghebeur, Erik and Shariatmadar, Keivan and De Cooman, Gert},
  issn         = {0165-0114},
  journal      = {FUZZY SETS AND SYSTEMS},
  language     = {eng},
  pages        = {74--88},
  title        = {Constrained optimization problems under uncertainty with coherent lower previsions},
  doi          = {10.1016/j.fss.2012.02.004},
  volume       = {206},
  year         = {2012}
}

@article{Erik-2090091,
  author       = {Quaeghebeur, Erik and Huntley, Nathan and Shariatmadar, Keivan and De Cooman, Gert},
  booktitle    = {Communications in Computer and Information Science},
  isbn         = {9783642317170},
  issn         = {1865-0929},
   language     = {eng},
 journal = {IPMU},
  location     = {Catania, Italy},
  pages        = {430--439},
  publisher    = {Springer},
  title        = {Maximin and maximal solutions for linear programming problems with possibilistic uncertainty},
  doi          = {10.1007/978-3-642-31718-7\_45},
  volume       = {4-299},
  year         = {2012}
}

@article{Rockafellar-2001,
  author       = {R. T. Rockafellar},
  journal      = {LECTURE NOTES},
  language     = {eng},
  pages        = {118},
  publisher    = {University of Washington},
  title        = {OPTIMIZATION UNDER UNCERTAINTY},
  year         = {2001}
}

@article{ContainerTerminal,
author = {Szpytko, Janusz},
journal = {Logistics and Transport},
number = {13},
pages = {107--116},
title = {{Automated Guided Vehicles Navigating Problem In Container Terminal}},
volume = {2},
year = {2011}
}

@incollection{TSP,
author = {Rajesh Matai and Surya Singh and Murari Lal Mittal},
title = {Traveling Salesman Problem: an Overview of Applications, Formulations, and Solution Approaches},
booktitle = {Traveling Salesman Problem},
publisher = {IntechOpen},
address = {Rijeka},
year = {2010},
editor = {Donald Davendra},
chapter = {1},
doi = {10.5772/12909},
url = {https://doi.org/10.5772/12909}
}

@INPROCEEDINGS{Changjoo2015, 
author={C. {Nam} and D. A. {Shell}}, 
booktitle={2015 IEEE International Conference on Robotics and Automation (ICRA)}, 
title={When to do your own thing: Analysis of cost uncertainties in multi-robot task allocation at run-time}, 
year={2015}, 
volume={}, 
number={}, 
pages={1249-1254}, 
doi={10.1109/ICRA.2015.7139351}, 
ISSN={1050-4729}, 
month={May}
}

@book{Tgal1979,
  title={Postoptimal analyses, parametric programming, and related topics},
  author={G{\'a}l, T.},
  isbn={9780070226791},
  lccn={78040091},
  series={Advanced Book Program - McGraw-Hill Book Company},
  url={https://books.google.be/books?id=8Zl-AAAAIAAJ},
  year={1979},
  publisher={McGraw-Hill International}
}

@INPROCEEDINGS{DeRyck2020b,
author = {{De Ryck}, M. and Versteyhe, M. and Shariatmadar, K.},
booktitle = {2020 6th International Conference on Mechatronics and Robotics Engineering, Barcelona, Spain},
title = {Methodology for a Gradual Migration from a Centralized towards a Decentralized Control in AGV Systems},
publisher = {IEEE Xplore® digital library},
year = {2020}
}

@Article{Ward1990,
author="Ward, James E. and Wendell, Richard E.",
title="Approaches to sensitivity analysis in linear programming",
journal="Annals of Operations Research",
year="1990",
month="Dec",
day="01",
volume="27",
number="1",
pages="3--38",
issn="1572-9338",
doi="10.1007/BF02055188",
url="https://doi.org/10.1007/BF02055188"
}

@article{LIN2007,
author = {Lin, Chi-Jen and Wen, Ue-Pyng},
title = {Sensitivity analysis of objective function coefficients of the assignment problem},
journal = {Asia-Pacific Journal of Operational Research},
volume = {24},
number = {02},
pages = {203-221},
year = {2007},
doi = {10.1142/S0217595907001115}
}

@Article{Matthias2019,
title = {Resource Management in Decentralized Industrial Automated Guided Vehicle Systems},
author = {De Ryck, Matthias and Versteyhe, Mark and Shariatmadar, Keivan},
journal={Journal of Manufacturing Systems},
ISSN = {0278-6125},
pages = {204--214},
volume = {54},
publisher = {Elsevier},
year = {2020}
}

@book{Bostelman2016,
author = {Bostelman, Roger and Messina, Elena},
year = {2016},
month = {04},
title = {Towards Development of an Automated Guided Vehicle Intelligence Level Performance Standard},
isbn = {978-0-8031-7633-1},
publisher = {West Conshohocken, PA: ASTM International},
doi = {10.1520/STP159420150054}
}

@article{Pecin2017a,
author = {Pecin, Diego and Contardo, Claudio and Desaulniers, Guy and Uchoa, Eduardo},
title = {New Enhancements for the Exact Solution of the Vehicle Routing Problem with Time Windows},
journal = {INFORMS Journal on Computing},
volume = {29},
number = {3},
pages = {489-502},
year = {2017},
doi = {10.1287/ijoc.2016.0744},
URL = {https://doi.org/10.1287/ijoc.2016.0744}   
}

@Article{Pecin2017b,
author={Pecin, Diego and Pessoa, Artur and Poggi, Marcus and Uchoa, Eduardo},
title={Improved branch-cut-and-price for capacitated vehicle routing},
journal={Mathematical Programming Computation},
year={2017},
month={Mar},
day={01},
volume={9},
number={1},
pages={61-100},
issn={1867-2957},
doi={10.1007/s12532-016-0108-8},
url={https://doi.org/10.1007/s12532-016-0108-8}
}

@Article{Zhang2019,
author={Zhang, Yu
and Baldacci, Roberto
and Sim, Melvyn
and Tang, Jiafu},
title={Routing optimization with time windows under uncertainty},
journal={Mathematical Programming},
year={2019},
month={May},
day={01},
volume={175},
number={1},
pages={263-305},
issn={1436-4646},
doi={10.1007/s10107-018-1243-y},
url={https://doi.org/10.1007/s10107-018-1243-y}
}
\newpage
\appendix
\section{Annex}
\label{Annex}
In the next two sections, we will give maximin and maximal solutions to the problem \eqref{spp} in two separate uncertainty models---linear previsions (or probability distributions) and vacuous previsions (or intervals).

\subsection{Linear prevision model}
\label{linsec}
Suppose that the uncertainty about the random variables $Y$ and $Z$ is given by cumulative distribution functions $F_Y$ and $F_Z$ respectively, which are (assumed to) be independent and $F_V$ is a joint distribution of the random variable $V=(Y,Z)$, then we can find maximin and maximal solutions as follows.

\subsubsection{Optimal solutions in linear prevision case}\label{opt:lin}
In this model, because of the property \ref{add-P}, the maximin and maximal solutions coincide and can be found by simply replacing the lower probability $\underline P$ in Equation \eqref{MAXIMIN} with a probability measure $P$,
\begin{equation}\label{LINPRE}
\argmax_{x\ge 0}(c^Tx-L)P(Yx\le Z)
\end{equation}
if $Y$ is invertible then we have:
\begin{align}\label{LINPRE}
\argmax_{x\ge 0}(c^Tx-L)P(Yx\le Z)&=\argmax_{x\ge 0}(c^Tx-L)\bigl(1-P(Yx\not\le Z)\bigr)\notag\\
&= \argmax_{x\ge 0}(c^Tx-L)\bigl(1-F_{Y^{-1}Z}(x)\bigr)
\end{align}
where $F$ is the cumulative distribution function. 

\subsubsection{Linear prevision properties}\label{LP-prop}
The conjugate operators $\lp$ and $\up$, defined for all gambles $g$ on $\coacts$ and related by $\up(g)=-\lp(-g)$, are characterised by the following three conditions:\footnote{The assumption that $\lp$ and $\up$ are defined on the whole of $\gambles{\coacts}$ is not a trivial one in general: so-called natural extension from a partial specification requires solving a linear programming problem \cite{Walley-1991}Chapter 3. However, for the cases examined in this paper, natural extension just requires calculating a (Choquet) integral in the most complex case, which is far less computationally demanding.}\\
\(\text{(Positivity)}\label{pos-lp}~\underline P(h)\ge\inf{h}~\text{or}~\overline P(h)\le\sup{h};~\text{(Positive Homogeneity)}\label{hom-lp}~\underline P(\lambda g)=\lambda\underline P(g)\) or \(\overline P(\lambda g)=\lambda\overline P(g);~\\\text{(Super/Sub-additivity)}\label{suadd-lp}~\underline P(g+h)\ge\underline P(g)+\underline P(h)~\text{or}~\overline P(g+h)\le\overline P(g)+\overline P(h)\)
where $g,h\in\mathcal{G(Y)}$ and $\lambda>0$. 

Their restriction to (indicators of) events is called coherent lower and upper probabilities.
In the behavioural interpretation of \cite{Walley-1991}, who follows the lead of \cite{DeFinetti-1975} in this regard, lower and upper previsions for gambles are again seen as prices: respectively the agent's supremum acceptable buying price and infimum acceptable selling price.
When the lower prevision coincides with the upper prevision, they are linear previsions see, e.g., \cite{Walley-1991}. All the uncertainty models we deal with fall into either category, but thanks to this framework, we can treat constrained (linear) optimisation problems with uncertain variables described by both types in the constraint specification in a unified way and with a unified interpretation. It also leaves the door open to similarly deal with problems involving other types of uncertainty models under the coherent lower prevision umbrella.
Lower (and upper) previsions can also be defined for indicators of events that are exactly equivalent to the lower (and upper) probability of the same events, that is, $\underline P(I_A)=\underline P(A)$. In one special case, when the lower prevision coincides with the upper prevision, then the third condition (Super/Sub-additivity) becomes (Additivity) condition and we have linear previsions see, e.g. \cite{Walley-1991}. In addition, it has shown \cite{Walley-1999} that working with a coherent lower prevision $\underline P$ is equivalent to working with a convex closed set of linear previsions (or probabilities) $\mathcal M(\underline P)$, which is set of dominating linear previsions by $\underline P$:
\begin{equation}
\mathcal M(\underline P):=\left\{P:\forall h\in\mathcal{G(Y)},P(h)\ge\underline P(h)\right\}.
\end{equation}
And vice versa, $\underline P$ is the lower envelope of the set $\mathcal M(\underline P)$:
\begin{equation}\label{l-env}
\underline P(h):=\min\left\{P(h):P\in\mathcal M(\underline P)\right\},\quad\forall h\in\mathcal{G(Y)}.
\end{equation} 

There are some other properties for coherent lower (and upper) previsions which are used further on in this paper:\\ 
\(\text{(Constant additivity)}\label{cadd-lp}~\underline P(h+\mu)=\underline P(h)+\mu;~\text{(Mixed super-additivity)}\label{msadd-lp}~\overline P(g+h)\ge \underline P(g)+\overline P(h);~\text{(Factorisation)}\label{factor-lp}~\underline P(f_1f_2)= \underline P\bigl(f_1\underline P(f_2)\bigr) \) 
  where $g,h\in\mathcal{G(Y)},~f_1>0,~f_1$ and $f_2$ are independent see, e.g. \cite{deCooman20111911} and $\mu\in\mathbb{R}$.

\subsection{Vacuous model}
\label{vacsec}
In this case, the uncertainty about $V=(Y,Z)$ is modelled by a coherent lower prevision $\underline P_V$ on $\mathcal{A}:=A\times B\subseteq\mathbb{R}^{m\times n}\times\mathbb{R}^m$ for a given $g=g(V)~\text{on}~\mathcal{A}$ as, 
\begin{equation}\label{maxivac}
\underline P_V(g):=\min_{v\in\mathcal{A}}~g(v)~~\text{equivalently,}~~\overline P_V(g):=\max_{v\in\mathcal{A}}~g(v)
\end{equation} 
where \(A:=\bigtimes_{k=1}^{m} \bigtimes_{l=1}^{n} A_{kl};\;A_{kl}:=[\underline y_{kl},\overline y_{kl}],~ B:=\bigtimes_{k=1}^{m} B_k;\;B_k:=[\underline z_k,\overline z_k].\)

\subsubsection{Maximin solutions in vacuous case}
\label{maximin-vacuous}
By combining these vacuous prevision definitions in \eqref{maxivac} with the Equation \eqref{MAXIMIN}, the maximin solutions become a classical linear programming problem: 
\begin{align}\label{maxisol-vac}
\argmax_{x\ge 0}\underline P( G_x) &= \quad\argmax_{x\ge 0}~\left[(c^Tx-L)\min_{(y,z)\in\mathcal A} I_{yx\le z}(x)\right]\notag\\
&=\hspace{-4mm} \argmax_{\{x\ge 0:\forall (y,z)\in\mathcal A,~yx\le z\}}\hspace{-3mm}c^Tx\notag\\ 
&=: \quad\argmax_{x\in\underline{\mathbb A}}\quad c^Tx
\end{align}
where $\underline{\mathbb A}:=\underset{(y,z)\in\mathcal{A}}{\bigcap} \left\{x\ge 0:yx\le z\right\}
$ is an \emph{inner feasibility space} and is calculated as follows:
\begin{align}
\underline{\mathbb A}:&=\underset{(y,z)\in\mathcal{A}}{\bigcap} \left\{x\in\mathbb{R}^n_{\ge 0}:\sum_{l=1}^ny_{kl}x_l\le z_k,~k=1,2,\dots,m\right\}\notag\\
&=\left\{x\in\mathbb{R}^n_{\ge 0}:(\forall (y,z)\in\mathcal{A})~\sum_{l=1}^ny_{kl}x_l\le z_k,~k=1,2,\dots,m\right\}\notag\\
&=\left\{x\in\mathbb{R}^n_{\ge 0}:\max_{y_{kl}\in A_{kl}}\sum_{l=1}^ny_{kl}x_l\le \min_{z_k\in B_k}z_k,~k=1,2,\dots,m\right\}\notag\\
&=\left\{x\in\mathbb{R}^n_{\ge 0}:\sum_{l=1}^n \overline y_{kl}x_l\le \underline{z_k},~k=1,2,\dots,m\right\}=:\left\{x\in\mathbb{R}^n_{\ge 0}:\overline Yx\le \underline Z\right\}
\end{align}
Thus, the solution, in this case, can be written as:
\begin{align}\label{maximinsol}
\max_{x\in\mathbb{R}^n}& \quad c^Tx\notag\\
\text{such that}&\quad\overline{\,Y}x \le \underline Z,~x\ge 0.
\end{align}
Indeed, if $\underline{\mathbb A}=\emptyset$, then the problem is infeasible.

\subsubsection{Maximal solutions in vacuous case}\label{maximal-vacuous}
We arrange the decision $x\ge 0$ in a (partial) order so that is not dominated by any other decisions $w\ge 0$. To do that, we divide the (decisions) space $(x,w)\in\mathbb{R}^{2n}$ into quadrants:

\begin{wrapfigure}{l}{1mm}
    \begin{tikzpicture}[xscale=1.5,yscale=.6]
    \draw (0,0) rectangle (1,1);
    \node at (.5,.5) {${x^+}\cap{w^-}$};
    \draw (0,0) rectangle (1,-1);
    \node at (.5,-.5) {${x^-}\cap{w^-}$};
    \draw (0,0) rectangle (-1,1);
    \node at (-.5,.5) {${x^+}\cap{w^+}$};
    \draw (0,0) rectangle (-1,-1);
    \node at (-.5,-.5) {${x^-}\cap{w^+}$};
  \end{tikzpicture}
  \end{wrapfigure}
  {
 \begin{align}\label{quadra-def} 
\hspace{3.2cm} x^+:=&\{(y,z)\in\mathcal A:yx\le z,~x\ge0\}\quad\;\;\;\, x^-:=\{(y,z)\in\mathcal A:yx\not\le z,~x\ge0\}\notag\\[3mm]
 \qquad w^+:=&\{(y,z)\in\mathcal A:yw\le z,~w\ge0\}\quad w^-:=\{(y,z)\in\mathcal A:yw\not\le z,~w\ge0\}
 \end{align}}\\
By considering the maximality definition \eqref{MAXIMAL} and all possible non-empty ($2^4-1 = 15$) cases where $\mathcal A$ is relative to these quadrants, we seek an expression for $\overline P(G_x-G_w)$, 
\begin{align*}
  \overline P(G_x-G_w)
  &= \begin{cases}
    c^Tx-L & \mathcal A\cap{x^+}\cap{w^-}\not=\emptyset,\\
    0 &\mathcal A\cap{x^+}=\emptyset\wedge\mathcal A\cap{w^-}\not=\emptyset,\\
    \max\{0,c^Tx-c^Tw\}  &\mathcal A\cap{x^+}\cap{w^-}=\emptyset \wedge\mathcal A\cap{x^+}\not=\emptyset \wedge\mathcal A\cap{w^-}\not=\emptyset,\\
    c^Tx-c^Tw &\mathcal A\cap{x^+}\not=\emptyset \wedge\mathcal A\cap{w^-}=\emptyset,\\
    L-c^Tw &\mathcal A\cap{x^+}=\emptyset \wedge\mathcal A\cap{w^-}=\emptyset.
  \end{cases}
\end{align*}
The first three cases are always non-negative, the fourth one can be positive or negative, and the last one is always negative. Therefore, we consider the last two cases avoiding that $x\ge 0$ is not maximal, for more details see \cite{Keivan-1225787, Erik-2030199}:
\begin{align}
 & \hspace{-4mm}\inf_{w\in\mathbb{R}^n}\overline P(G_x-G_w)\ge0 \notag\\
  \Leftrightarrow  & \underline{\mathbb A}=\emptyset\vee\group{ x\in\overline{\mathbb A}\wedge c^Tx\ge\max_{w\in\underline{\mathbb A}} c^Tw=c^Tx_m}
\end{align}

\hspace{-4mm}where, $x_m$ is the maximin solution and $\overline{\mathbb A}:=\underset{(y,z)\in\mathcal{A}}{\bigcup} \left\{x\ge 0:yx\le z\right\}$ is \emph{outer feasibility space} and is calculated as follows:
\begin{align}
\overline{\mathbb A}:&=\underset{(y,z)\in\mathcal{A}}{\bigcup} \left\{x\in\mathbb{R}^n_{\ge 0}:\sum_{l=1}^ny_{kl}x_l\le z_k,~k=1,2,\dots,m\right\}\notag\\
&=\left\{x\in\mathbb{R}^n_{\ge 0}:(\exists (y,z)\in\mathcal{A}):\sum_{l=1}^ny_{kl}x_l\le z_k,~k=1,2,\dots,m\right\}\notag\\
&=\left\{x\in\mathbb{R}^n_{\ge 0}:\min_{y_{kl}\in A_{kl}}\sum_{l=1}^ny_{kl}x_l\le \max_{z_k\in B_k}z_k,~k=1,2,\dots,m\right\}\notag\\
&=\left\{x\in\mathbb{R}^n_{\ge 0}:\sum_{l=1}^n \underline y_{kl}x_l\le \overline{z_k},~k=1,2,\dots,m\right\}=:\left\{x\in\mathbb{R}^n_{\ge 0}:\underline Yx\le \overline Z\right\}.
\end{align}
Therefore, if $\underline{\mathbb A}\not=\emptyset$, then the maximal solutions become a classical feasibility problem: 
 \begin{equation}\label{maximalsol}
 \left\{x\in\mathbb{R}^n_{\ge 0}:x\in\overline{\mathbb{A}}~\text{and}~c^Tx\ge c^Tx_m\right\}\equiv\left\{x\in\mathbb{R}^n_{\ge 0}:\underline Yx\le \overline Z~\text{and}~c^Tx\ge c^Tx_m \right\}.
\end{equation}
One of the interesting properties of these results is that the solutions in both criteria---maximinity and maximality---do not depend on $L$. 
The uncertainty model, in the incarnations described above, is a special case of the much more general coherent upper and lower previsions e.g., see for details and terminology \cite{Walley-1991,Miranda-2008}.
 
\clearpage

%\section*{Author Biography}
%{\includegraphics[width=76pt,height=66pt]{me_new_2018.png}}
%\begin{biography}{\textbf{Keivan Shariatmadar}, a senior researcher at M-Group, Campus Bruges, KU Leuven University, Belgium. We focus on modelling advanced uncertainty and application in several areas, e.g., smart manufacturing, fleet management, Industry 4.0, smart city, human/operator/driver behaviour modelling, and optimisation.}
%\end{biography}

\end{document}